\documentclass[12pt,oneside]{article}
\usepackage[centertags]{amsmath}
\usepackage{amsfonts}
\usepackage{amssymb}
\usepackage{amsthm}
\usepackage{mathrsfs}
\usepackage{setspace}
\usepackage[papersize={8.5in,11in}]{geometry}
\newtheorem{thm}{Theorem}[section]
\newtheorem{cor}[thm]{Corollary}
\newtheorem{lem}[thm]{Lemma}

\newtheorem{con}[thm]{Conjecture}

\numberwithin{equation}{section}

\numberwithin{equation}{section}
\author {Ya\c{s}ar Polato\~{g}lu, Oya Mert$^{\ast}$ and Asena \c{C}etinkaya}
\title {On The Coefficient Conjecture of Clunie and Sheil-Small
\footnotetext{\textit{\textnormal{2010} AMS Mathematics Subject Classification:}
30C45 \\ \textit{Key words and phrases:} Harmonic mapping, Schwarz function, subordination, coefficient estimate, distortion theorem.\\
Corresponding Author$^{\ast}$
}}
\date{ }

\newcommand{\de}{{\mathbb D}}
\newcommand{\ce}{{\mathbb C}}
\begin{document}
\maketitle

\begin{abstract}
In this paper, we first prove relation between analytic and co-analytic part of the class  harmonic univalent functions $\mathcal{S}_{\mathcal{H}}(\mathcal{S}):=\left\{f=h+\overline g|\ h\in \mathcal{S}\right\} $ by means of second dilatation is constant. Next, we verify the coefficient conjecture of Clunie and Sheil-Small for class of  harmonic univalent functions. Finally, we obtain distortion bounds of this class. 
\end{abstract}

\maketitle

\section{Introduction}
Planar harmonic univalent mappings and related functions have applications in the diverse fields of Engineering, Physics, Electronics, Medicine, and other branches of applied mathematical sciences. For example, E. Heinz \cite{HEN} in 1952 used such mappings in the study of the Gaussian curvature of nonparametric minimal surfaces over the unit disc. Harmonic univalent mappings have attracted the serious attention of complex analysts after the appearance of a basic paper by Clunie and Sheil-Small \cite{Clunie1984} in 1984. The works of these researchers and several others (e.g. see \cite{SCHO, SCHO2, Sheil}) gave rise to several interesting problems, conjectures, and questions in harmonic univalent theory.

Let  $\mathcal{H}$ be the family of continouos complex-valued harmonic functions $f  =h+\overline{g}$ defined in the open unit disc  $\mathbb{D}:=\{z: |z|<1\}$, where $h$ and $g$ has the  power series expansion $h(z)=\sum_{n=0}^\infty a_{n}{z^n}$ and $g(z)=\sum_{n=1}^\infty b_{n}{z^n}.$ Here $h$ the analytic  and $g$ the co-analytic part of $f$. The Jacobian $J_{f}$ of a function $f$ is
$$J_{f}(z)=|f_z(z)|^2-|f_{\overline z }(z)|^2=|h'(z)|^2-|g'(z)|^2,$$ 
where $f=h+\overline{g}$ is the harmonic function in $\de$. A harmonic mapping $f=h+\overline{g}$  defined in $\de$ is sense-preserving if $|h'|>|g'|$ in $\de$; sense-reversing if  $|h'|<|g'|$ in $\de$. An analytic univalent function is a special case of an sense-preserving harmonic univalent function. For analytic function $f$, it is well known that  $J_{f}(z)\neq 0$ if and only if $f$ is locally univalent at $z$. For harmonic functions we have the following useful result due to Lewy:  
\begin{thm}\label{teo1}$\cite{LEW}$ If $f=h+\overline{g}$ is a complex-valued harmonic function that is locally univalent $\de\subset\ce$, then its jacobian is $J_f(z)\neq 0$ for every $z\in\de$.
\end{thm}
\begin{thm}\label{teo2}$\cite{AHL}$ An analytic function is injective in some neighborhood of a point if and only if its derivative does not vanish at that point.
\end{thm}
Note that $f=h+\overline{g}$  is locally univalent and sense-preserving in  $\mathbb{D}$ if and only if the second dilatation $|w|<1$ in $\mathbb{D}$ (see \cite{LEW}). We let $\mathcal{S_H}$ be a subclass of functions $f$ in $\mathcal{H}$ that are univalent and sense-preserving harmonic mappings $f=h + \overline{g}$ defined in $\de$, normalized by the conditions $f(0)=0$ and $f_z(0)=1$ in $\de$. Thus a function $f=h + \overline{g}$ in $\mathcal{S}_{\mathcal{H}}$ has the representation
\begin{equation}\label{equa11}
		f(z)=z+\sum^{\infty}_{n=2}a_{n}z^{n}+\sum^{\infty}_{n=1}b_{n}\overline{z}^{n},\quad z\in\de.
\end{equation} 
It follows from the sense-preserving property that $|b_1|<1.$ If we restrict with an extra condition of normalization that $f_{\overline{z}}(0)=b_1=0$, then the class $\mathcal{S}_{\mathcal{H}}$ is denoted by $\mathcal{S}^0_{\mathcal{H}}$.  We observe that for $g(z)\equiv 0$ in $\de$, the class $\mathcal{S_H}$ reduces to the class $\mathcal{S}$ of normalized analytic univalent functions in $\de$. Thus, $\mathcal{S}\subset\mathcal{S}^0_{\mathcal{H}}\subset\mathcal{S}_{\mathcal{H}}$. For history of  families $\mathcal{S}$ and $\mathcal{S}_{\mathcal{H}}$, the reader may refer to \cite{DUR, POM}. 

In 1984, Clunie and Sheil-Small \cite{Clunie1984}  investigated the class $\mathcal{S_H}$ as well as its geometric subclasses and obtained some coeffcient bounds. Clunie and Sheil-Small \cite{Clunie1984}  discovered a result for the family $\mathcal{S}^0_{\mathcal{H}}$, analogous to the Koebe function which is in the class $\mathcal{S}$. In fact, they constructed the harmonic Koebe function $k_0=h +\overline{g}\in\mathcal{S}^0_{\mathcal{H}}$ defined by
\begin{equation}\label{equa13}
	k_0(z)=\frac{z-\frac{1}{2}z^2+\frac{1}{6}z^3}{(1-z)^3}+\overline{\frac{\frac{1}{2}z^2+\frac{1}{6}z^3}{(1-z)^3}}.
\end{equation}
Additionally, Clunie and Sheil-Small \cite{Clunie1984} obtained the following harmonic analogues of the Bieberbach conjecture for the family $\mathcal{S}^0_{\mathcal{H}}$.
\begin{con}\label{con1}$\cite{Clunie1984}$ If $f=h+\overline{g}$ given by (\ref{equa11}) is in the class $\mathcal{S}^0_{\mathcal{H}}$, then	
	\begin{enumerate}
		\item[(i)]	$\big||a_n|-|b_n|\big|\leq n,\quad n\geq2.$
		\item[(ii)]	$|a_n|\leq \frac{(2n+1)(n+1)}{6}, \quad n\geq2.$
		\item[(iii)] $|b_n|\leq \frac{(2n-1)(n-1)}{6}, \quad n\geq2.$
	\end{enumerate}
	Equality occurs for $k_0$ which given in (\ref{equa13})	
\end{con}

Let $\Omega$ be the family of functions $\phi$ which are analytic on $\de$, and satisfy the conditions $\phi(0)=0, |\phi(z)|<1$ for all $z\in\de$. If $f_1$ and $f_2$ are analytic functions  on $\de$, then we say that $f_1$ is subordinate to $f_2$ written as  $f_1\prec f_2$, if there exists a Schwarz function $\phi\in \Omega$ such that $f_1(z)=f_2(\phi(z))$.
\begin{lem}\label{lem1}$\cite{JACK}$(Jack's Lemma)  Let $\phi(z)$ be regular in the open unit disc $\mathbb{D}$ with $\phi(0)=0$, $|\phi(z)|<1$.  Then if $|\phi(z)|$ attains its maximum value on the circle $|z|=r$ at the point $z_0$, one has $z_0\phi'(z)={k}\phi(z_0),$ $k\geq 1$.
\end{lem}

In this present paper, we will investigate the class of harmonic univalent functions given as 
$$\mathcal{S}_{\mathcal{H}}(\mathcal{S}):=\big\{f=h+\overline{g}|  h\in\mathcal{S}\big\}.$$
We first introduce relation between analytic and co-analytic part of the class harmonic univalent functions by means of second dilatation is constant. In Theorem \ref{teo4}, we prove the Conjecture \ref{con1} part (i) and further, in Theorem \ref{teo5} the distortion bounds of $f$ will be proven.
\section{Main Results}
\begin{thm}\label{teo3} Let $f=h+\overline g$ given by (\ref{equa11}) be an element of $\mathcal{S}_{\mathcal{H}}(\mathcal{S})$, then 
	\begin{equation}
		\displaystyle|g(z)|<|h(z)|.
	\end{equation}
\end{thm}
\begin{proof}
	Since $f=h+\overline{g}\in \mathcal{S}_{\mathcal{H}}(\mathcal{S})$,  then we can write  
	$$w(z)=\frac{g'(z)}{h'(z)}\Rightarrow w(0)=b_1.$$
	By condition of Schwarz function and subordination,
	$$ \phi(z)=\frac{w(z)-w(0)}{1-\overline{w(0)}w(z)}$$ 
	can be written
	$$w(z)=\frac{b_1+\phi(z)}{1+\overline{b_1}\phi(z)}\Leftrightarrow w(z)\prec \frac{b_1+z}{1+\overline{b_1}z}.$$
	On the other hand the linear transformation $\displaystyle\frac{b_1+z}{1+\overline{b_1}z} $ maps $|z|=r$ onto the disc with the centre 
	$$C(r)=\bigg(\displaystyle\frac{(1-r^2)\Re b_1}{1-|b_1|^2r^2},\displaystyle\frac{(1-r^2)\Im b_1}{1-|b_1|^2r^2}\bigg),$$ and with the radius $$\rho(r)=\displaystyle\frac{(1-|b_1|^2)r}{1-|b_1|^2r^2}.$$
	Then, we have 
	$$\bigg|\displaystyle\ w(z)-\frac{|b_1|(1-r^2)}{1-|b_1|^2r^2}\bigg|\leq\displaystyle\frac{(1-|b_1|^2)r}{1-|b_1|^2r^2}.$$		
	Using the subordination principle, then we can write
	\begin{equation}
		\displaystyle w(\mathbb{D}_{r})=\left\{\frac{g'(z)}{h'(z)}:\bigg|\displaystyle\frac{ g'(z)}{h'(z)}-\frac{b_1(1-r^2)}{1-|b_1|^2r^2}\bigg|\leq\displaystyle\frac{(1-|b_1|^2)r}{1-|b_1|^2r^2},r\in\mathcal(0,1)\right\}.
	\end{equation}	
In order to verify Schwarz function conditions, we define the function $\phi$ by
\begin{equation}
	\displaystyle \frac{g(z)}{h(z)}=\frac{b_1+\phi(z)}{1+\overline b_1\phi(z)}.
\end{equation}
Note that $\phi$ is a well defined analytic function and $\phi(0)=0$. We now need to show that $\phi$ satisfies the condition $|\phi(z)|<1$ for all $z\in \mathbb{D}$. Taking the derivative from (2.3), we obtain that 
\begin{equation}
	\displaystyle w(z)=\frac{g'(z)}{h'(z)}=\frac{b_1+\phi(z)}{1+\overline b_1\phi(z)}+\frac{(1-|b_1|^2)z\phi'(z)}{(1+\overline b_1\phi(z))^2}\frac{h(z)}{zh'(z)}.
\end{equation}
Using the Koebe transformation of the function $h$ (see \cite{POM}), the equality (2.4) can be written in the form 
\begin{equation}
	\displaystyle w(z)=\frac{g'(z)}{h'(z)}=\bigg(\frac{b_1+\phi(z)}{1+\overline b_1\phi(z)}+\frac{(1-|b_1|^2)z\phi'(z)}{(1+\overline b_1\phi(z))^2}\frac{1}{1-|z|^2}\frac{|h(z)|}{|z|}{e^{i\theta}}\bigg).
\end{equation}
Assume to the contrary that there exists a point  $z_1\in\mathbb{D}_{r}$ such that $|\phi(z_1)|=1$. In view of Lemma \ref{lem1}, the equality (2.5) gives  
$$\displaystyle w(z_1)=\frac{g(z_1)}{h(z_1)}=\Big(\frac{b_1+\phi(z_1)}{1+\overline b_1\phi(z_1)}+\frac{(1-|b_1|^2) k\phi(z_1)}{(1+\overline b_1\phi(z_1))^2}\frac{1}{1-|z_1|^2}\frac{|h(z_1)|}{|z_1|}\Big)\notin w(\mathbb{D}_{r})$$
This is a contradiction with (2.2) and hence $|\phi(z)|<1$ for all $z\in\de$. Therefore, we have 
$$\frac{g(z)}{h(z)}\prec\frac{b_1+z}{1+\overline b_1(z)}\iff \frac{g(z)}{h(z)}=\frac{b_1+\phi(z)}{1+\overline b_1\phi(z)}=w(z).$$
This shows that $g(z)=h(z)w(z)\Rightarrow |g(z)|<|h(z)|.$
\end{proof}
\begin{cor} Let $f=h+\overline g\in\mathcal{S}_{\mathcal{H}}(\mathcal{S})$, then $g(z)=ch(z)$, $|c|<1$ and complex.
\end{cor}
\begin{proof} Using Theorem 2.1 and definition of second dilatation, we write 
	$$\frac{g(z)}{h(z)}=w(z)=\frac{g'(z)}{h'(z)}\Rightarrow \log g(z)=\log ch(z) \Rightarrow g(z)=ch(z).$$
\end{proof}
\begin{cor} Let $f=h+\overline g$ be an element of $\mathcal{S}_{\mathcal{H}}(\mathcal{S})$, then the second dilatation,  $\displaystyle w(z)=\frac{g'(z)}{h'(z)}$ is constant.
\end{cor}
\begin{proof} Using Corollary 2.2, we have $$g(z)=ch(z)\Rightarrow g'(z)=ch'(z)\Rightarrow w(z)=\frac{g'(z)}{h'(z)}=c$$
	We note that harmonic mapping with the constant dilatation had been introduced by Antti Rasila \cite{RAS}.	
\end{proof}
\begin{cor} Let $f=h+\overline g$ be a harmonic mapping, then $$f\in \mathcal{S}_{\mathcal{H}}(\mathcal{S})\iff h\in \mathcal{S}.$$
This corollary is a simple consequence of Theroem \ref{teo1}, Theorem \ref{teo2} and Theorem \ref{teo3}.	
\end{cor}
\begin{thm}\label{teo4} Let $f=h+\overline g$ given by (\ref{equa11}) be an element of $\mathcal{S}_{\mathcal{H}}(\mathcal{S})$, then 
\begin{enumerate}
\item[(a)] $|b_n| \leq |c|\displaystyle\frac{n(n+1)}{2},$
\item[(b)] $\big||a_n|-|b_n|\big|\leq n,$
\end{enumerate}
where $n\geq2.$ These ineqaulities are sharp.
\end{thm}
\begin{proof} Since $g(z)=ch(z)$, $|c|<1$, then we have $b_1=c, ...,b_n=ca_n$, $n\geq2$. Since $h\in\mathcal{S}$,  deBranges  Theorem \cite{BRA} states that $|a_n|\leq n$ for every $n\geq2.$ Therefore, we obtain that 
	$$|b_1|+|b_2|+|b_3|+...+|b_n|\leq |c|(1+|a_2|+|a_3|+...+|a_n|)$$  $$\qquad\qquad\qquad\quad\quad\leq |c|(1+2+3...+n)$$ $$\Rightarrow |b_n|\leq |c|\frac{n(n+1)}{2}$$
	This proves the part (a) of theorem. On the other hand, since $|c|<1$ and $n\geq 2$, let
	$$\displaystyle |c|\frac{n+1}{2}>0\Rightarrow -|c|\frac{n+1}{2}<0$$
	$$\quad\Rightarrow 1-|c|\frac{n+1}{2}\leq 1\Rightarrow n(1-|c|\frac{n+1}{2})\leq n$$
	$$\qquad\quad\Rightarrow \displaystyle  |a_n|-|b_n|\leq |a_n-b_n|\leq n (1-|c|\frac{n+1}{2})\leq n$$
	$$\Rightarrow \displaystyle ||a_n|-|b_n||\leq n$$
This proves the part (b) of theorem. This proof verify the Conjecture \ref{con1} part (i).
\end{proof}
\begin{thm}\label{teo5} Let $f=h+\overline g$ given by (\ref{equa11}) be an element of $\mathcal{S}_{\mathcal{H}}(\mathcal{S})$, then 
 $$dist (0,\partial f)\geq \frac{1}{4}(1-|c|).$$ 
\end{thm}
\begin{proof}
	Since  $f=h+\overline g\in \mathcal{S}_{\mathcal{H}}(\mathcal{S})$, then we have distortion bound of the function $h$ as
	\begin{equation}
		\frac{1-r}{(1+r)^3}\leq |h'(z)|\leq\frac{1+r}{(1-r)^3}, 
	\end{equation}
	and since the function $f$ is sense-preserving, we have
	$$(|h'(z)|-|g'(z)|)|dz|\leq |df|\leq (|h'(z)|+|g'(z)|)|dz|,$$
	\begin{equation}
		(1-|c|)|h'(z)||dz|\leq |df|\leq (|1+|c|)|h'(z)||dz|.
	\end{equation}
	Considering (2.6) and (2.7) together, we obtain 
	$$(1-|c|)\int \frac{1-r}{(1+r)^3}dr\leq |f|\leq (1+|c|)\int \frac{1+r}{(1-r)^3}dr.$$
	Calculating of above integral, we get
	\begin{equation}
		\frac{(1-|c|)r}{(1+r)^2}\leq |f|\leq \frac{(1+|c|)r}{(1-r)^2}.
	\end{equation}
	Therefore, \\
	$$dist(0, \partial f)=\lim_{r=|z|\to 1}\inf |f|\geq\lim_{r=|z|\to 1} \frac{(1-|c|)|z|}{(1+|z|)^2}=\frac{1}{4}(1-|c|).$$
\end{proof}

\footnotesize
\vspace{1cm}

\textsc{Ya\c{s}ar Polato\~{g}lu}\\
Department of Mathematics and Computer Sciences\\
\.{I}stanbul K\"{u}lt\"{u}r University, \.{I}stanbul, Turkey\\ 
e-mail: y.polatoglu@iku.edu.tr\\

\textsc{Oya Mert}\\
Department of Basic Sciences\\
Alt{\i}nba\c{s} University, \.{I}stanbul, Turkey\\ 
e-mail:oya.mert@altinbas.edu.tr  \\

\textsc{Asena \c{C}etinkaya}\\
Department of Mathematics and Computer Sciences\\
\.{I}stanbul K\"{u}lt\"{u}r University, \.{I}stanbul, Turkey\\ 
e-mail: asnfigen@hotmail.com \\


\begin{thebibliography}{1}
	\bibitem{AHL} L. V. Ahlfors, \textit{Complex Analysis}, 3 rd edition, Mc Graw-Hill New York, 1979.
	\bibitem{Clunie1984} J. Clunie, T. Sheil-Small, \textit{Harmonic univalent functions,} Ann. Acad. Sci. Fenn. Ser. A I
	Math. {\bf 9} (1984), 3-25.
	\bibitem{BRA} L. deBranges, \textit{A proof of the Bieberbach Conjecture}, Acta. Math. {\bf 154}(1-2) (1985), 137-152.
	\bibitem{DUR} P. Duren, \textit{Harmonic Mappings in the Plane}, Cambridge University Press, 2004.
	\bibitem{HEN} E. Heinz, \textit{\"{U}ber die L\"{o}sungen der Minimalfl\"{a}chengleichung,} (German) Nachr. Akad. Wiss. Göttingen. Math.-Phys. Kl. Math.-Phys.-Chem. Abt. (1952), 51-56.
	\bibitem{SCHO} W. Hengartner, G. Schober, \textit{Univalent harmonic functions}, Trans. Amer. Math. Soc. {\bf 299}(1) (1987), 1-31.
	\bibitem{JACK} I. S. Jack, \textit{Functions starlike and convex of order $\alpha$,} J. London Math. Soc. {\bf 2}(3) (1971), 469-474.
	\bibitem{LEW} H. Lewy, \textit{On the non-vanishing of the Jacobian in certain one-to-one mappings}, Bull. Amer. Math. Soc. {\bf 42} (1936), 689-692.
	\bibitem{POM} C. R. Pommerenke, \textit{Univalent Functions}, Vandenhoeck, Ruprecht in Gottingen, 1975.
	\bibitem{RAS} A. Rasilla, \textit{Mappings problems and harmonic univalent mappings}, Helsinki Analysis Seminar, (2009), 11-16.
	\bibitem{SCHO2} G. Schober, \textit{Planar harmonic mappings,} In: Lecture Notes in Math. 478 (1975), Springer Verlag, New York.
	\bibitem{Sheil} T. Sheil-Small, \textit{Constants for planar harmonic mappings,} J. London Math. Soc. {\bf 42}(2) (1990), 237-248.
	\end{thebibliography}
\end{document}